\documentclass[leqno,12pt]{article}
\usepackage{latexsym}
\usepackage{amsmath}
\usepackage{amssymb}
\usepackage{enumerate}

\usepackage{theorem}
\newtheorem{theorem}{Theorem}[section]

\newtheorem{lemma}[theorem]{Lemma}
\newtheorem{corollary}[theorem]{Corollary}

\theorembodyfont{\rmfamily}
\newtheorem{proof}{\textmd{\textit{Proof.}}}

\newtheorem{remark}[theorem]{Remark}
\newtheorem{example}[theorem]{Example}
\newtheorem{definition}[theorem]{Definition}

\makeatletter

\@addtoreset{equation}{section}
\makeatother

\newcommand{\qedd}{\hfill \Box}
\newcommand{\ve}{\varepsilon}
\newcommand{\lra}{\longrightarrow}
\newcommand{\ora}{\overrightarrow}
\newcommand{\ola}{\overleftarrow}
\newcommand{\wt}{\widetilde}
\newcommand{\wh}{\widehat}
\newcommand{\ol}{\overline}
\newcommand{\N}{\ensuremath{\mathbb{N}}}
\newcommand{\R}{\ensuremath{\mathbb{R}}}
\newcommand{\bH}{\ensuremath{\mathbb{H}}}
\newcommand{\Sph}{\ensuremath{\mathbb{S}}}

\newcommand{\cG}{\ensuremath{\mathcal{G}}}

\newcommand{\cN}{\ensuremath{\mathcal{N}}}
\newcommand{\cT}{\ensuremath{\mathcal{T}}}
\newcommand{\cO}{\ensuremath{\mathcal{O}}}

\newcommand{\cR}{\ensuremath{\mathcal{R}}}

\def\diam{\mathop{\mathrm{diam}}\nolimits}
\def\rad{\mathop{\mathrm{rad}}\nolimits}

\def\Cut{\mathop{\mathrm{Cut}}\nolimits}

\title{Grove\,-\,Shiohama type sphere theorem\\ in Finsler geometry
}

\author{Kei KONDO}
\date{\today}
\begin{document}
\maketitle

\begin{abstract}
From radial curvature geometry's standpoint, 
we prove a sphere theorem of the Grove-Shiohama type for a certain class of 
compact Finsler manifolds.
\end{abstract}

\section{Introduction}\label{sec1}

Beyond a doubt, one of the most beautiful theorems in global Riemannian geometry is 
the diameter sphere theorem of Grove and Shiohama \cite{GS}. 
In their proof, Toponogov's comparison theorem (TCT) was very first applied {\bf seriously} together with the critical point theory, introduced by themselves, of distance functions. That is, if a point $x$ in a complete Riemannian manifold is a critical point of the distance function $d_p$ to a point $p$ in the manifold, then $x$ is the cut point of $p$, and hence $d_p$ is not differentiable at $x$. 
However, they overcame the analytical obstruction by applying the original TCT to the triangle $\triangle (pxy)$ 
with the interior angle $\angle (pxy) \le \pi/2$. That's the point, i.e., 
they took the manifold into their hands by directly drawing segments on it.\par 
Our purpose of this article is {\em to prove a sphere theorem of the Grove-Shiohama type 
for a certain class of forward complete 
Finsler manifolds whose radial flag curvatures are bounded below by $1$}. 
Of course, our major tools to prove it are a TCT for such a class and the critical point theory, more precisely, Gromov's isotopy lemma (\cite{G}). Such a TCT is easily proved by modifying the TCT established in \cite{KOT1} (see Section \ref{sec2} in this article), 
and the isotopy lemma holds from a similar argument to the Riemannian case. 
The fact that, compared with the Riemannian case, 
there are {\bf few} theorems on the relationship between the topology and the curvature of a Finsler manifold is the {\em worthy of note}. E.g., Shen's finiteness theorem (\cite{Shvol}), 
Rademacher's quarter pinched sphere theorem (\cite{Ra}), 
the Gauss-Bonnet formula for surfaces with 
non-constant indicatrix of Itoh, Sabau, and Shimada (\cite{ISS}), 
Ohta's splitting theorem (\cite{Ospl}), and the finiteness of topological type and 
a diffeomorphism theorem to Euclidean spaces of the author with Ohta and Tanaka 
in \cite{KOT2}.

\medskip

To state our sphere theorem of the Grove-Shiohama type in Finsler case, 
we will introduce several notions in the geometry and radial curvature geometry: 
Let $(M, F, p)$ denote a pair of a forward complete, connected,
$n$-dimensional $C^\infty$-Finsler manifold $(M,F)$ with a base point $p \in M$,
and $d: M \times M \lra [0, \infty)$ denote the distance function induced from $F$.
Remark that the {\em reversibility} $F(-v)=F(v)$ is not assumed in general,
and hence $d(x, y) \not= d(y,x)$ is allowed.

For a local coordinate $(x^i)^{n}_{i=1}$ of an open subset $\cO \subset M$,
let $(x^i, v^j)_{i,j=1}^{n}$ be the coordinate of the tangent bundle $T\cO$ over $\cO$ such that 
\[ v:= \sum_{j = 1}^{n} v^j \frac{\partial}{\partial x^j}\Big|_{x},
 \qquad \ x \in \cO. \]
For each $v \in T_xM \setminus \{0\}$, the positive-definite $n \times n$ matrix 
\[
\big( g_{ij} (v) \big)_{i,j= 1}^{n}:= 
\left(
\frac{1}{2} \frac{\partial^2 (F^2)}{\partial v^i \partial v^j}(v)
\right)_{i, j = 1}^{n}
\]
provides us the Riemannian structure $g_{v}$ of $T_x M$ by 
\[
g_{v} \left(
\sum_{i = 1}^{n} a^i \frac{\partial}{\partial x^i}\bigg|_{x}, 
\sum_{j = 1}^{n} b^j \frac{\partial}{\partial x^j}\bigg|_{x}
\right) 
:=
\sum_{i,j = 1}^{n} g_{ij} (v) a^ib^j. 
\]
This is a Riemannian approximation of $F$ in the direction $v$.
For two linearly independent vectors $v, w \in T_{x} M \setminus \{0\}$,
the {\em flag curvature} is defined by 
\[
K_M (v, w) := \frac{g_{v} (R^{v} (w, v)v, w)}{g_{v} (v, v) g_{v} (w, w) - g_{v} (v, w)^2},
\]
where $R^{v}$ denotes the curvature tensor induced from the Chern connection.
Remark that $K_M (v, w)$ depends on the \emph{flag} $\{sv + tw\,|\, s, t \in \R\}$, and also on the \emph{flag pole} $\{sv\,|\, s> 0\}$.

Given $v,w \in T_xM \setminus \{0\}$, define the \emph{tangent curvature} by 
\[
\cT_M(v, w) := g_X\big( D^Y_Y Y(x) - D^X_Y Y(x), X(x) \big), 
\]
where the vector fields $X,Y$ are extensions of $v,w$,
and $D_{v}^{w}X(x)$ denotes the covariant derivative of $X$ by $v$ with reference vector $w$. 
Independence of $\cT_M(v,w)$ from the choices of $X,Y$ is easily checked.
Note that $\cT_M \equiv 0$ if and only if $M$ is of \emph{Berwald type}
(see \cite[Propositions~7.2.2, 10.1.1]{Sh}).
In Berwald spaces, for any $x,y \in M$, 
the tangent spaces $(T_xM, F|_{T_xM})$ and $(T_yM, F|_{T_yM})$
are mutually linearly isometric (cf.~\cite[Chapter~10]{BCS}).
In this sense, $\cT_M$ measures the variety of tangent Minkowski normed spaces.

\medskip

Let $\wt{M}$ be a complete $2$-dimensional Riemannian manifold,
which is homeomorphic to $\R^{2}$ if $\wt{M}$ is non-compact,
or to $\Sph^{2}$ if $\wt{M}$ is compact. 
Fix a base point $\tilde{p} \in \wt{M}$.
Then, we call the pair $(\wt{M}, \tilde{p})$ a {\em model surface of revolution} 
if its Riemannian metric $d\tilde{s}^2$ is expressed 
in terms of the geodesic polar coordinate around $\tilde{p}$ as 
\[
d\tilde{s}^2 = dt^2 + f(t)^2d \theta^2, \qquad 
(t,\theta) \in (0,a) \times \Sph_{\tilde{p}}^1, 
\]
where $0<a\le \infty$, $f : (0, a) \lra \R$ denotes a positive smooth function 
which is extensible to a smooth odd function around $0$, 
and $\Sph^{1}_{\tilde{p}} := \{ v \in T_{\tilde{p}} \wt{M} \,|\, \| v \| = 1 \}$. 
Define the {\em radial curvature function} $G: [0,a) \lra \R$
such that $G(t)$ is the Gaussian curvature at $\wt{\gamma}(t)$,
where $\wt{\gamma}:[0,a) \lra \wt{M}$ is any (unit speed) meridian
emanating from $\tilde{p}$.
Note that $f$ satisfies the differential equation 
$f''+Gf=0$ with initial conditions $f(0) = 0$ and $f'(0) = 1$. 
It is clear that, if $f (t) = t, \sin t, \sinh t$, then $\wt{M} = \R^2, \Sph^2, \bH^2(-1)$, respectively. 
We call $(\wt{M}, \tilde{p})$ a {\em von Mangoldt surface}
if $G$ is non-increasing on $[0,a)$. 
A round sphere is the only compact, `smooth' von Mangoldt surface, 
i.e., $f$ satisfies $\lim_{t \uparrow a}f'(t)=-1$. If a von Mangoldt surface has the property $a <\infty$ and 
if it is not a round sphere, then $\lim_{t \uparrow a} f(t)=0$ and $\lim_{t \uparrow a} f'(t)>-1$.
Therefore, such a surface $(\wt{M}, \tilde{p})$ has a singular point, say $\tilde{q} \in \wt{M}$,
at the maximal distance from $\tilde{p} \in \wt{M}$ such that $d(\tilde{p},\tilde{q})= a$, and hence 
$\wt{M}$ is an Alexandrov space. Its shape can be understood as a `balloon'.
\begin{example}{\rm (\cite[Example 1.2]{K})}
Set $f (t) := \frac{t(1 - t)(1 + t)}{11t^{4} - 25t^{2} + 18}$.
Then, the compact surface of revolution 
$(\wt{M}, \tilde{p})$ with $d\tilde{s}^2 = dt^2 +  f(t)^2d \theta^2$ is of von Mangoldt type and 
has a singular point at $t = 1$. In particular, $-\infty < \lim_{t \uparrow 1} G (t) < 0$. 
\end{example}

Paraboloids and $2$-sheeted hyperboloids are typical examples of non-compact 
von Mangoldt surfaces.
An atypical example of such a surface is as follows.

\begin{example}{\rm (\cite[Example 1.2]{KT1})}
Set $f (t) := e^{- t^{2}} \tanh t$ on $[0, \infty)$.
Then, the non-compact surface of revolution 
$(\wt{M}, \tilde{p})$ with $d\tilde{s}^2 = dt^2 +  f(t)^2d \theta^2$ is of von Mangoldt type, 
and $G$ changes the sign.
Indeed, $\lim_{t \downarrow 0}G (t) = 8$ and $\lim_{t \to \infty}G (t) = - \infty$.
\end{example}

We say that a Finsler manifold $(M, F, p)$ has the
{\em radial flag curvature bounded below by that of a model surface of revolution 
$(\wt{M}, \tilde{p})$} if, 
along every unit speed minimal geodesic $\gamma: [0,l) \lra M$ 
emanating from $p$, we have
\[
K_{M} \big(\dot{\gamma}(t), w \big) \ge G (t)
\]
for all $t \in [0, l)$ and $w \in T_{\gamma(t)}M$ linearly independent to $\dot{\gamma}(t)$. 
Also, we say that $(M, F, p)$ has the {\em radial tangent curvature bounded below 
by a constant $\delta \in (-\infty, 0]$}
if, along every unit speed minimal geodesic $\gamma: [0,l) \lra M$ emanating from $p$, 
\[
\cT_M(\dot{\gamma}(t), w) \ge \delta
\]
for all $w \in T_{\gamma(t)}M$.

\medskip

Our main result is now stated: 

\begin{theorem}\label{2013_02_12_thm}
Let $(M, F, p)$ be a compact connected $n$-dimensional $C^{\infty}$-Finsler manifold 
whose radial flag curvature is bounded below by $1$ and radial tangent curvature is equal to $0$. Assume that 
\begin{enumerate}[{\rm (1)}]
\item 
$F(w)^2 \ge g_v(w,w)$ 
for all $x \in B^+_{\frac{\pi}{2}} (p)$, $v \in \cG_p(x)$, and $w \in T_xM$, 
\item 
$g_v(w, w) \ge F(w)^2$ for all $x \in M \setminus \ol{B_{\frac{\pi}{2}}^+ (p)}$,
$v \in \cG_p (x)$ and $w  \in T_xM$, 
\item
the reverse curve $\bar{c}(s):=c(a-s)$ of $c$ is geodesic and $L_{\rm m} (c) \le \rad_p$ 
for all minimal geodesic segments $c:[0,a] \lra M \setminus \{p\}$. 
\end{enumerate}
If $\rad_p> \pi / 2$, then $M$ is homeomorphic to the sphere $\Sph^n$. 
\end{theorem}

\medskip

In Theorem \ref{2013_02_12_thm}, we set $B_r^+ (p):=\{ x \in M \,|\, d(p,x)<r \}$, 
\begin{equation}\label{G_p}
\mathcal{G}_p (x) := \{ \dot{\gamma}(l) \in T_{x}M \,|\,
 \text{$\gamma$ is a minimal geodesic segment from $p$ to $x$} \},
\end{equation}
where $l:=d(p,x)$, 
$L_{\rm{m}}(c):= \int^a_0 \max\{ F(\dot{c}),F(-\dot{c}) \} \,ds$, and $\rad_p:= \sup_{x \in M}d(p,x)$. 
The assumptions (1) and (2) are the $2$-uniform convexities 
with the sharp constant (see \cite{Ouni}), 
but only for {\bf special points} $x$ and {\bf directions} $v$, respectively. 
The sharpness means that (1) and (2) hold for {\bf all} $(x,v) \in TM \setminus \{0\}$
only if $F$ is Riemannian. 
One may construct non-Riemannian spaces satisfying (1) and (2) (see \cite{KOT1}). 
The geodesic property on $\bar{c}$ in the (3) and $\cT_M(\dot{\gamma}(t), w) = 0$ 
just only imply $g_{\dot{\gamma}} (D^{\dot{\gamma}}_{\dot{c}}\dot{c}, \dot{\gamma}) = 0$. 
Note that $D^{\dot{\gamma}}_{\dot{c}}\dot{c} \not=0$ in general. 
The (3) holds, if $F$ is reversible and $\rad_p = \diam (M):= \sup_{x,\,y \in M} d(x, y)$. 
Note that $\diam (M)\le \pi$ from the Bonnet-Myers theorem (\cite[Theorem 7.7.1]{BCS}). 
If $F$ is of Berwald type, the geodesic property on $\bar{c}$ in the (3) and $\cT_M(\dot{\gamma}(t), w)=0$ are automatically satisfied. In particular, Theorem \ref{2013_02_12_thm} contains 
the diameter sphere theorem as a special case. 

\begin{remark}In the (3) of Theorem \ref{2013_02_12_thm}, 
we can replace $L_{\rm m} (c) \le \rad_p$ with the following weaker assumption:
\[
L_{\rm m} (c)
\begin{cases}
< \pi \quad \text{for $c$ satisfying $c ( [0,a] ) \cap (M \setminus B_{\frac{\pi}{2}}^+ (p)) \not= \emptyset$,}\\[2mm]
\le \rad_p \quad \text{for $c$ emanating from $q \in \partial B_{\rad_p}^+ (p)$ 
to any point in $B_{\frac{\pi}{2}}^+ (p)$.}
\end{cases}
\]
Note that $\partial B_{\rad_p}^+ (p) = \{q\}$ (see Lemma \ref{2013_02_17_lem3.3}).
\end{remark}

\begin{remark} Probably, one can generalize 
Theorem \ref{2013_02_12_thm} to a wider class of metrics than those described in it, 
that is, by employing a von Mangoldt surface of the balloon type 
satisfying $f'(\rho) =0$ for unique $\rho \in (0, a)$, $\lim_{t \uparrow a} f(t)=0$, 
$\lim_{t \uparrow a} f'(t)>-1$, and $\rad_p > \rho$. 
Of course, more assumptions would be demanded to generalize it than 
those in Theorem \ref{2013_02_12_thm}. In the Riemannian case, see \cite[Theorem A]{KO1}.
\end{remark}

\medskip\noindent{\textit{Acknowledgements.}}
I would like to thank Professor M. Tanaka for helpful discussions.

\section{TCTs}\label{sec2}

To prove Theorem \ref{2013_02_12_thm}, we need Toponogov's comparison 
theorems (TCT) in Finsler geometry. In \cite{KOT1}, we recently established a TCT for a certain class of Finsler manifolds whose radial flag curvatures are bounded below by that of a von Mangoldt surface. 
In this section, we modify the TCT in the case where a model surface is the unit sphere. 

\subsection{Angles, triangles, and a counterexample}

Let $(M,F,p)$ be a forward complete, connected $C^\infty$-Finsler manifold with a base point $p \in M$,
and denote by $d$ its distance function. It follows from the Hopf-Rinow 
theorem that the forward completeness guarantees that any 
two points in $M$ can be joined by a minimal geodesic segment. 
Owing to $d(x,y) \neq d(y,x)$ generally, we need a distance with the symmetric property to 
define the `angles': Define 
\[
d_{\rm{m}} (x, y) := \max\{d(x, y), d(y, x)\}.
\]
Since $|d(p,x) -d(p,y)| \le d_{\rm{m}} (x, y)$, 
we may define the angles with respect to $d_{\rm{m}}$ as follows.

\begin{definition}{\bf (Angles)}\label{2013_02_13_def2.1}
Let $c :[0,a] \lra M$ be a unit speed minimal geodesic segment
(i.e., $F(\dot{c}) \equiv 1$) with $p \not\in c([0,a])$. 
The \emph{forward} and the \emph{backward angles}
$\ora{\angle}(pc(s)c(a))$, $\ola{\angle}(pc(s)c(0)) \in [0,\pi]$ at $c(s)$ are defined via
\begin{align*}
\cos \ora{\angle}\big( pc(s)c(a) \big) &:= -\lim_{h \downarrow 0} 
 \frac{d(p,c(s+h)) -d(p, c(s))}{d_{\rm{m}} (c(s), c(s + h))} \quad \text{for $s \in [0,a)$}, \\
\cos \ola{\angle}\big( pc(s)c(0) \big) &:= \lim_{h \downarrow 0} 
 \frac{d(p, c(s)) - d (p, c(s-h))}{d_{\rm{m}} (c(s -h), c(s))} \quad \text{for $s \in (0, a]$}.
\end{align*}
\end{definition}

\begin{remark}\label{2013_02_13_rem2.1}
The limits in Definition \ref{2013_02_13_def2.1} exist in $[-1,1]$ (see \cite[Lemma~2.2]{KOT1}).
\end{remark}

\begin{definition}{\bf (Forward triangles)}\label{2013_02_13_def2.2}
For three distinct points $p, x, y \in M$,
\[ \triangle (\ora{px}, \ora{py}) := (p, x, y; \gamma, \sigma, c) \]
will denote the \emph{forward triangle} consisting of unit speed minimal geodesic segments
$\gamma$ emanating from $p$ to $x$, $\sigma$ from $p$ to $y$, and $c$ from $x$ to $y$.
Then the corresponding \emph{interior angles} $\ora{\angle}x, \ola{\angle}y$
at the vertices $x$, $y$ are defined by 
\[
\ora{\angle}x := \ora{\angle}\big( p c(0) c(a) \big), \qquad
 \ola{\angle}y := \ola{\angle}\big( p c(a) c(0) \big), 
\]
respectively, where $a:=d(x,y)$.
\end{definition}

\begin{definition}{\bf (Comparison triangles)}\label{2013_02_13_def2.3}
Fix a model surface of revolution $(\wt{M}, \tilde{p})$. 
Given a forward triangle $\triangle (\ora{px}, \ora{py})= (p, x, y; \gamma, \sigma, c)\subset M$,
a geodesic triangle $\triangle (\tilde{p}\tilde{x} \tilde{y}) \subset \wt{M}$ is called
its \emph{comparison triangle} if
\[ \tilde{d}(\tilde{p}, \tilde{x}) = d(p, x), \qquad 
\tilde{d}(\tilde{p},\tilde{y}) = d(p, y), \qquad 
\tilde{d}(\tilde{x},\tilde{y}) = L_{\rm{m}}(c) \]
hold, where $L_{\rm{m}}(c)= \int^{d(x,\,y)}_0 \max\{ F(\dot{c}),F(-\dot{c}) \} \,ds$.
\end{definition}

There are many forward triangles admitting their comparison triangles, but 
TCT {\bf does not} always hold 
for all of them: 

\begin{example}(\cite{KO2})\label{2013_02_13_exa2.1} 
For an even number $q$,  let $M$ be $\R^2$ with the $\ell^q$-norm. 
Then, $M$ is Minkowskian. 
Take a forward triangle $\triangle (\ora{px}, \ora{py}) \subset M$, 
where $p :=(0, 0), x :=(1,0), y:=(0,1) \in M$, and let $c (t):= (1-t, t)$ denote the side 
of $\triangle (\ora{px}, \ora{py})$ joining $x$ to $y$. 
Assume that $q$ is sufficiently large. Then, we observe that 
both angles $\ora{\angle} x$ and $\ola{\angle} y$ are nearly $0$, respectively. 
We are able to think of $(\R^2, \tilde{p})$ as a reference surface for $M$, 
because flag curvature $K_M \equiv 0$. It is clear that $\triangle (\ora{px}, \ora{py})$ admits its comparison triangle $\triangle (\tilde{p}\tilde{x} \tilde{y}) \subset \R^2$. 
Since $\triangle (\ora{px}, \ora{py})$ is nearly equilateral, $\triangle (\tilde{p}\tilde{x} \tilde{y})$ is too. 
Hence, $\ora{\angle} x < \angle \tilde{x}$ and $\ola{\angle} y < \angle \tilde{y}$ hold. 
Therefore, TCT does not hold for the $\triangle (\ora{px}, \ora{py})$. 
\end{example}

\subsection{Modified TCTs}

From Example \ref{2013_02_13_exa2.1}, we understand that 
some strong conditions are demanded to establish a TCT in Finsler geometry. 
Taking this into account, we have the following:

\begin{theorem}{\rm (\cite[Theorem 1.2]{KOT1})}\label{2013_02_13_thm2.1}
Assume that $(M, F, p)$ is a forward complete, connected $C^{\infty}$-Finsler manifold
whose radial flag curvature is bounded below by that of a von Mangoldt surface 
$(\wt{M}, \tilde{p})$ satisfying $f'(\rho) =0$ for unique $\rho \in (0, \infty)$. 
Let $\triangle (\ora{px}, \ora{py}) = (p, x, y; \gamma, \sigma, c) \subset M$ be a forward triangle 
satisfying that, for some open neighborhood $\cN(c)$ of $c$,
\begin{enumerate}[$(1)$]
\item
$c([0,d(x,y)]) \subset M \setminus \ol{B^+_{\rho} (p)}$,
\item 
$g_v(w,w) \ge F(w)^2$ for all $z \in \cN(c)$, $v \in \cG_p(z)$ and $w \in T_zM$,
\item
$\cT_{M}(v,w) = 0$ for all $z \in \cN(c)$, $v \in \cG_p(z)$ and $w \in T_zM$,
and the reverse curve $\bar{c}(s):=c(d(x,y)-s)$ of $c$ is also geodesic.
\end{enumerate}
If such $\triangle (\ora{px}, \ora{py})$ admits a comparison triangle $\triangle (\tilde{p}\tilde{x} \tilde{y})
\subset \wt{M}$,
then we have $\ora{\angle} x \ge \angle \tilde{x}$ and $\ola{\angle} y \ge \angle \tilde{y}$.
\end{theorem}

\begin{remark}
In Theorem \ref{2013_02_13_thm2.1}, $f'(t) < 0$ on $(\rho, \infty)$. 
\end{remark}

\begin{corollary}\label{2013_02_13_cor2.1}
Assume that $(M, F, p)$ is a compact connected $C^{\infty}$-Finsler manifold
whose radial flag curvature is bounded below by $1$. 
Let $\triangle (\ora{px}, \ora{py}) = (p, x, y; \gamma, \sigma, c) \subset M$ be a forward triangle 
satisfying that, for some open neighborhood $\cN(c)$ of $c$,
\begin{enumerate}[$(1)$]
\item
$c([0,d(x,y)]) \subset M \setminus \ol{B^+_{\frac{\pi}{2}} (p)}$,
\item 
$g_v(w,w) \ge F(w)^2$ for all $z \in \cN(c)$, $v \in \cG_p(z)$ and $w \in T_zM$,
\item
$\cT_{M}(v,w) = 0$ for all $z \in \cN(c)$, $v \in \cG_p(z)$ and $w \in T_zM$,
and the reverse curve $\bar{c}(s):=c(d(x,y)-s)$ of $c$ is also geodesic.
\end{enumerate}
If such $\triangle (\ora{px}, \ora{py})$ admits a comparison triangle $\triangle (\tilde{p}\tilde{x} \tilde{y})$ 
in $(\Sph^2, \tilde{p})$, 
then we have $\ora{\angle} x \ge \angle \tilde{x}$ and $\ola{\angle} y \ge \angle \tilde{y}$. 
Here, $(\Sph^2, \tilde{p})$ denotes the unit sphere, i.e., its Riemannian metric $d\tilde{s}^2$ is expressed as $d\tilde{s}^2 = dt^2 + f(t)^2d \theta^2$, $(t,\theta) \in (0,\pi) \times \Sph_{\tilde{p}}^1$, 
such that $f(t) = \sin t$.
\end{corollary}

\begin{proof}
Since $f'(t) = \cos t < 0$ on $(\pi/2, \pi)$ and $f'(\pi/2) =0$ for unique $\pi/2 \in (0, \pi)$, 
the corollary is immediate from Theorem \ref{2013_02_13_thm2.1}.
$\qedd$
\end{proof}

\begin{lemma}\label{2013_02_13_lem2.1}
Assume that $(M, F, p)$ is a compact connected $C^{\infty}$-Finsler manifold
whose radial flag curvature is bounded below by $1$. 
Let $\triangle (\ora{px}, \ora{py}) = (p, x, y; \gamma, \sigma, c) \subset M$ be a forward triangle 
satisfying that, for some open neighborhood $\cN(c)$ of $c$,
\begin{enumerate}[$(1)$]
\item
$c([0,d(x,y)]) \subset B^+_{\frac{\pi}{2}} (p) \setminus \{p\}$,
\item 
$F(w)^2 \ge g_v(w,w)$ for all $z \in \cN(c)$, $v \in \cG_p(z)$ and $w \in T_zM$,
\item
$\cT_{M}(v,w) = 0$ for all $z \in \cN(c)$, $v \in \cG_p(z)$ and $w \in T_zM$,
and the reverse curve $\bar{c}(s):=c(d(x,y)-s)$ of $c$ is also geodesic.
\end{enumerate}
If such $\triangle (\ora{px}, \ora{py})$ admits a comparison triangle $\triangle (\tilde{p}\tilde{x} \tilde{y})$ 
in $(\Sph^2, \tilde{p})$, 
then we have $\ora{\angle} x \ge \angle \tilde{x}$ and $\ola{\angle} y \ge \angle \tilde{y}$. 
\end{lemma}

\begin{proof}
Set $\lambda:= \max \{F(w), F(-w)\}$. The assumption (2) yields $\lambda^2 \ge g_v(w,w)$. 
Hence, one can prove this lemma by the almost similar argument as that in \cite{KOT1}. 
See Section \ref{sec4} in this article for a detailed explanation of that.
$\qedd$
\end{proof}

\begin{remark}
As a corollary to Lemma \ref{2013_02_13_lem2.1}, 
a TCT holds for forward complete, connected $C^{\infty}$-Finsler manifolds $(M, F, p)$ 
whose radial flag curvatures are bounded below by a {\bf non-positive constant}, 
because, roughly speaking, the index forms on the models are positive. 
In the TCT, the assumptions (2) and (3) in Lemma \ref{2013_02_13_lem2.1} are demanded 
(owing to our theory), but we do not need to assume the (1) from their metric properties of the models. 
We shall discuss its applications, to extend the classic theorems 
in global Riemannian geometry, elsewhere. 
\end{remark}

\section{Proof of Theorem \ref{2013_02_12_thm}}\label{sec3}

Let $(M,F,p)$ be the same as that in Theorem \ref{2013_02_12_thm}. Hence, 
our model surface as a reference surface is the unit sphere $(\Sph^2, \tilde{p})$, i.e., its Riemannian metric $d\tilde{s}^2$ is expressed as $d\tilde{s}^2 = dt^2 + f(t)^2d \theta^2$, $(t,\theta) \in (0,\pi) \times \Sph_{\tilde{p}}^1$, such that $f(t) = \sin t$.

\begin{lemma}\label{2013_02_16_lem3.1}
The set $\Sph^2 \setminus B_{t}(\tilde{p})$ is strictly convex for all $t \in (\pi/2,\pi)$, i.e., 
for any distinct two points $\tilde{x},\tilde{y} \in \partial B_{t}(\tilde{p})$ 
and minimal geodesic segment $\tilde{c}:[0,a] \lra \Sph^2$ between them, we have
$\tilde{c} ( (0,a) ) \subset \Sph^2 \setminus \ol{B_{t}(\tilde{p})}$, 
where $a:= \tilde{d} (\tilde{x},\tilde{y})$
\end{lemma}

\begin{proof}
Use the second variation formula. 
$\qedd$
\end{proof}

Hereafter, by the Bonnet-Myers theorem (\cite[Theorem 7.7.1]{BCS}), we may assume, without loss of generality, 
\begin{equation}\label{2013_05_19_Sec3.1}
\rad_p < \pi.
\end{equation}

\begin{lemma}{\bf (Key Lemma)}\label{2013_02_17_lem3.1}
For any distinct two points $x, y \in M \setminus \ol{B_{\frac{\pi}{2}}^+ (p)}$,
then
\[
c ( [0,d(x, y)] ) \cap \partial B_{\frac{\pi}{2}}^+ (p) = \emptyset
\]
holds for all minimal geodesic segments $c$ emanating from $x$ to $y$.
\end{lemma}

\begin{proof}
Suppose that
$c ([0, d(x, y)]) \cap \partial B_{\frac{\pi}{2}}^+ (p) \not= \emptyset$
for some minimal geodesic segment $c$ emanating from $x$ to $y$. 
Then, we consider five cases:

{\bf Case 1:} Assume that there exist $s_0, s_1, s_2 \in [0,d(x, y))$ with 
$0 \le s_0 < s_1 < s_2$ such that  
\[
c([s_0, s_1))  \subset M \setminus \ol{B_{\frac{\pi}{2}}^+ (p)}, \qquad 
c([s_1, s_2]) \subset \partial B_{\frac{\pi}{2}}^+ (p).
\]
For sufficiently small $\ve >0$ with $\ve < s_1-s_0$, 
take the forward triangle $\triangle (\ora{pc(s_0)}, \ora{pc(s_1 -\ve)})$ $\subset M$. 
Note that $c([s_0, s_1-\ve]) \subset M \setminus \ol{B_{\frac{\pi}{2}}^+ (p)}$. 
Since $d(p, c(s_0)) > \pi/ 2$ and $d(p, c(s_1-\ve)) > \pi/ 2$, we have, 
by the assumption and \eqref{2013_05_19_Sec3.1}, that 
\begin{align*}
|d(p, c(s_0)) - d(p, c(s_1-\ve))| 
&\le d_{\rm m} (c(s_0), c(s_1-\ve))\\ 
&\le L_{\rm m} (c) < \pi < d(p, c(s_0)) + d(p, c(s_1-\ve)),
\end{align*}
and hence 
$\triangle (\ora{pc(s_0)}, \ora{pc(s_1 -\ve)})$ admits 
a comparison triangle $\triangle (\tilde{p}\wt{c(s_0)}\wt{c(s_1-\ve)})\subset \Sph^2$. 
By Corollary \ref{2013_02_13_cor2.1}, we have 
$\ola{\angle} (pc(s_1-\ve)c(s_0)) \ge \angle \wt{c(s_1-\ve)}$.  
It follows from \cite[Proposition 2.1]{TS} that 
\[
\frac{\pi}{2} = \ola{\angle} (pc(s_1)c(s_0)) \ge \lim_{\ve \downarrow 0}\ola{\angle} (pc(s_1-\ve)c(s_0)) 
\ge \angle \wt{c(s_1)}.
\]
Set 
$\triangle (\tilde{p}\wt{c(s_0)}\wt{c(s_1)}) := \lim_{\ve \downarrow 0}\triangle (\tilde{p}\wt{c(s_0)}\wt{c(s_1-\ve)})$, 
and let $\wt{\mu}:[0, \tilde{d}(\wt{c(s_0)}, \wt{c(s_1)})] \lra \Sph^2$ denote the side of $\triangle (\tilde{p}\wt{c(s_0)}\wt{c(s_1)})$ joining $\wt{c(s_0)}$ to $\wt{c(s_1)}$. 
If $\angle \wt{c(s_1)} = \pi/2$, then
\[
\wt{\mu}( [ 0, \tilde{d}(\wt{c(s_0)}, \wt{c(s_1)}) ] ) \subset \partial B_{\frac{\pi}{2}} (\tilde{p})
\] 
because $\partial B_{\frac{\pi}{2}} (\tilde{p})$ is geodesic. 
This contradicts $\tilde{d}(\tilde{p}, \wt{c(s_0)})> \pi/ 2$. 
If $\angle \wt{c(s_1)} < \pi/2$, 
then there exists $a \in ( 0, \tilde{d}(\wt{c(s_0)}, \wt{c(s_1)}) )$ 
such that $\wt{\mu}(a) \in \partial B_{\frac{\pi}{2}} (\tilde{p})$. 
This contradicts the structure of the cut locus 
of $\Sph^2$ because $\angle (\wt{\mu}(a) \tilde{p} \wt{c(s_1)}) < \pi$ and $\partial B_{\frac{\pi}{2}} (\tilde{p})$ is geodesic.

{\bf Case 2:} Assume that there exist $s_3, s_4, s_5 \in (0,d(x, y)]$ with 
$0 < s_3 < s_4 < s_5$ such that  
\[
c([s_3, s_4])  \subset \partial B_{\frac{\pi}{2}}^+ (p), \qquad 
c((s_4, s_5]) \subset M \setminus \ol{B_{\frac{\pi}{2}}^+ (p)}.
\]
Consider the forward triangle $\triangle (\ora{pc(s_4 + \ve)}, \ora{pc(s_5)}) \subset M$, 
where $\ve >0$ is sufficiently small with $\ve < s_5 -s_4$. Applying the similar limit 
argument in Case 1 to $\triangle (\ora{pc(s_4 + \ve)}, \ora{pc(s_5)})$, 
we have the triangle $\triangle (\tilde{p}\wt{c(s_4)}\wt{c(s_5)}) 
:= \lim_{\ve \downarrow 0}\triangle (\tilde{p}\wt{c(s_4 + \ve)}\wt{c(s_5)})$ 
satisfying $\angle \wt{c(s_4)} \le \pi/2$. The angle condition yields 
the same contradiction as that in Case 1.

{\bf Case 3:} Assume that there exist $s_0, s_1, s_2 \in [0,d(x, y)]$ with 
$s_0 < s_1 < s_2$ such that  
\[
c ([0, d(x, y)]) \cap \partial B_{\frac{\pi}{2}}^+ (p) = \{ c(s_1)\}, \qquad 
c((s_0, s_1)), c((s_1, s_2)) \subset M \setminus \ol{B_{\frac{\pi}{2}}^+ (p)}.
\]
Then, we get a contradiction from the same argument as Case 1, or Case 2. 

{\bf Case 4:} Assume that there exist $s_0, s_1 \in (0,d(x, y))$ with 
$s_0 < s_1$ such that  
\[
c((s_0, s_1))  \subset B_{\frac{\pi}{2}}^+ (p) \setminus \{p\}, \quad 
c(s_0), c(s_1) \in \partial B_{\frac{\pi}{2}}^+ (p), \quad 
\]
and that 
\begin{equation}\label{2013_02_17_lem3.2-1}
\ora{\angle} (pc(s_0)c(s_1)) < \frac{\pi}{2}, \quad \ola{\angle} (pc(s_1)c(s_0)) < \frac{\pi}{2}.  
\end{equation}
Take a subdivision $r_0 := s_0 < r_1 < \cdots < r_{k -1} <r_k :=  s_1$ 
of $[s_0, s_1]$ such that 
$\triangle (\ora{p c (r_{i-1})}, \ora{p c (r_i)})$ 
admits a comparison triangle 
$\wt{\triangle}^i := \triangle (\tilde{p} \wt{c (r_{i-1})} \wt{c (r_i)}) \subset \Sph^2$
for each $i = 1,2, \ldots, k$. 
Applying Lemma \ref{2013_02_13_lem2.1} to $\triangle (\ora{p c (r_{i-1})}, \ora{p c (r_i)})$, 
but for each $i = 2, 3, \ldots, k-1$, we have 
\begin{equation}\label{2013_02_17_lem3.2-2}
\ora{\angle} c (r_{i-1})  
\ge 
\angle\big( \tilde{p} \wt{c (r_{i-1})} \wt{c (r_i)} \big), \qquad 
\ola{\angle}c (r_i)
\ge 
\angle \big( \tilde{p} \wt{c (r_i)} \wt{c (r_{i-1})} \big).
\end{equation}
For sufficiently small $\ve, \delta >0$ with $\ve < r_1-r_0$ and $\delta < r_k -r_{k-1}$, 
take two forward triangles 
$\triangle (\ora{pc(r_0 + \ve)}, \ora{pc(r_1)}), \triangle (\ora{pc(r_{k-1})}, 
\ora{pc(r_k-\delta)}) \subset M$. 
Note that these two triangles admit their comparison triangles 
$\wt{\triangle}_\ve := \triangle (\tilde{p} \wt{c(r_0 + \ve)} \wt{c(r_1)}), 
\wt{\triangle}_\delta := \triangle (\tilde{p} \wt{c(r_{k-1})} \wt{c(r_k-\delta)}) \subset \Sph^2$, 
respectively. 
Without loss of generality, we may assume 
$\wt{\triangle}^1 = \lim_{\ve \downarrow 0}\wt{\triangle}_\ve$ and 
$\wt{\triangle}^k = \lim_{\delta \downarrow 0}\wt{\triangle}_\delta$ 
because $\lim_{\ve \downarrow 0}\wt{\triangle}_\ve$ and $\lim_{\delta \downarrow 0}\wt{\triangle}_\delta$ are 
isometric to $\wt{\triangle}^1$ and $\wt{\triangle}^k$, respectively. 
By Lemma \ref{2013_02_13_lem2.1}, 
$
\ora{\angle} c (r_0 + \ve)  
\ge 
\angle\big( \tilde{p} \wt{c (r_0 + \ve)} \wt{c (r_1)} \big)$,  
$\ola{\angle}c (r_1)
\ge 
\angle \big( \tilde{p} \wt{c (r_1)} \wt{c (r_0 + \ve)} \big)$, and that 
$\ora{\angle} c (r_{k-1})  
\ge 
\angle\big( \tilde{p} \wt{c (r_{k-1})} \wt{c (r_k -\delta)} \big)$, 
$
\ola{\angle}c (r_k -\delta)
\ge 
\angle \big( \tilde{p} \wt{c (r_k -\delta)} \wt{c (r_{k-1})} \big)$. 
Hence, it follows from \eqref{2013_02_17_lem3.2-1} and \cite[Proposition 2.1]{TS} that 
\begin{equation}\label{2013_02_17_lem3.2-3}
\frac{\pi}{2} > \ora{\angle} c(r_0) 
\ge \lim_{\ve \downarrow 0}\ora{\angle} c (r_0 + \ve) 
\ge \angle\big( \tilde{p} \wt{c (r_0)} \wt{c (r_1)} \big), \quad 
\ola{\angle}c (r_1) \ge \angle\big( \tilde{p} \wt{c (r_1)} \wt{c (r_0)} \big),
\end{equation}
and that 
\begin{equation}\label{2013_02_17_lem3.2-4}
\ora{\angle} c (r_{k-1})  
\ge 
\angle\big( \tilde{p} \wt{c (r_{k-1})} \wt{c (r_k)} \big), \ 
\frac{\pi}{2} > \ola{\angle}c (r_k) \ge \lim_{\delta \downarrow 0}\ola{\angle}c (r_k -\delta)
\ge 
\angle \big( \tilde{p} \wt{c (r_k)} \wt{c (r_{k-1})} \big). 
\end{equation}
Starting from $\wt{\triangle}^1$, we inductively draw a geodesic triangle 
$\wt{\triangle}^{i + 1} \subset \Sph^2$ which is adjacent to $\wt{\triangle}^i$ 
so as to have a common side
$\tilde{p} \wt{c (r_i)}$, where 
$0 := \theta (\wt{c (r_0)}) \le \theta (\wt{c (r_1)}) \le \cdots \le \theta (\wt{c (r_k)})$.
Since $\ola{\angle} c (r_i) + \ora{\angle} c (r_i) \le \pi$ for each $i = 1,2, \ldots, k-1$,  
we obtain, by \eqref{2013_02_17_lem3.2-2}, \eqref{2013_02_17_lem3.2-3}, \eqref{2013_02_17_lem3.2-4}, 
\begin{equation}\label{2013_02_17_lem3.2-5}
\angle\big( \tilde{p} \wt{c (r_i)} \wt{c (r_{i-1})} \big) + 
\angle\big( \tilde{p} \wt{c (r_i)} \wt{c (r_{i+1})} \big) 
\le \pi.
\end{equation}
Let $\wh{\xi}:[0, L_{\rm m} (c|_{[s_0,\,s_1]})] \lra \Sph^2$ denote the broken geodesic segment 
consisting of minimal geodesic segments from $\wt{c (r_{i-1})}$ to $\wt{c (r_i)}$,
$i=1,2,\ldots,k$.
Set $\wh{\xi} (s) := (t(\wh{\xi} (s)), \theta (\wh{\xi} (s)))$. 
By \eqref{2013_02_17_lem3.2-5}, we have the unit speed, 
but not necessarily minimal in this moment, geodesic 
$\wt{\eta} :[0, a] \lra \Sph^2$ emanating from $\wt{c (r_0)}$ to $\wt{c (r_k)}$ 
and passing under $\wh{\xi} ( [0, L_{\rm m} (c|_{[s_0,\,s_1]})] )$, i.e., 
$\theta (\wt{\eta}) \in [0, \theta (\wt{c (r_k)})]$ on $[0, a]$ 
and $t(\wh{\xi} (s)) > t(\wt{\eta} (u))$ for all $(s,u) \in (0,L_{\rm m} (c|_{[s_0,\,s_1]})) \times (0, a)$ with 
$\theta (\wh{\xi} (s)) = \theta (\wt{\eta} (u))$. 
Since $a \le L_{\rm m} (c|_{[s_0,\,s_1]}) < \pi$ by the assumption and \eqref{2013_05_19_Sec3.1}, 
$\wt{\eta}$ is minimal with $\angle (\wt{\eta}(0)\,\tilde{p}\,\wt{\eta}(a)) < \pi$. 
This contradicts the structure of the cut locus 
of $\Sph^2$ because $\partial B_{\frac{\pi}{2}} (\tilde{p})$ is geodesic.

{\bf Case 5:} Assume that $c$ is passing through $p$. Take a sequence $\{c_i:[0,l_i] \lra M \setminus \{p\}\}_{i\in \N}$ of minimal geodesic segments $c_i$ emanating from $x = c_i(0)$ 
convergent to $c$. Applying the same argument as that in Case 4 to each $c_i$ for sufficiently large $i$, 
we get a contradiction. Note that $\lim_{i\to\infty}L_{\rm m} (c_i) =L_{\rm m} (c) \le \rad_p$, but $x, y \in M \setminus \ol{B_{\frac{\pi}{2}}^+ (p)}$.\par 
Therefore, $c \big( [0,d(x, y)] \big) \cap \partial B_{\frac{\pi}{2}}^+ (p) = \emptyset$ 
holds for all minimal geodesic segments $c$ emanating from $x$ to $y$.
$\qedd$
\end{proof}

\begin{lemma}\label{2013_02_17_lem3.2}
The set $M \setminus B^+_{\frac{\pi}{2}}(p)$ is convex.
\end{lemma}

\begin{proof}
Take any distinct two points $x, y \in M \setminus \ol{B^+_{\frac{\pi}{2}}(p)}$, and let 
$c$ denote a minimal geodesic segment emanating from $x$ to $y$. 
Since $|d(p, x) - d(p, y)| \le L_{\rm m} (c) <  d(p, x) + d(p, y)$, 
the forward triangle $\triangle (\ora{px}, \ora{py}) \subset M$ admits 
its comparison triangle $\triangle (\tilde{p}\tilde{x}\tilde{y})\subset \Sph^2$. 
Thanks to Lemma \ref{2013_02_17_lem3.1}, we may apply Corollary \ref{2013_02_13_cor2.1} to 
$\triangle (\ora{px}, \ora{py})$. Combining Lemma \ref{2013_02_16_lem3.1} with Corollary \ref{2013_02_13_cor2.1}, we get the assertion.
$\qedd$
\end{proof}

\begin{lemma}\label{2013_02_17_lem3.3}
The function $d(p,\,\cdot\,)$ attains its maximum at a unique point $q \in M$.
In particular, $M \setminus B^+_{\frac{\pi}{2}}(p)$ is a topological disk.
\end{lemma}

\begin{proof} 
Suppose that there exist two distinct points $x,y \in \partial B^+_{\rad_p}(p)$.
Then, the forward triangle $\triangle (\ora{px}, \ora{py}) \subset M$ admits 
its comparison triangle $\triangle (\tilde{p}\tilde{x}\tilde{y})\subset \Sph^2$.
Let $c:[0,d(x,y)] \lra M$ and $\tilde{c}:[0,L_{\rm m} (c)] \lra \Sph^2$ be sides 
of $\triangle (\ora{px}, \ora{py})$ and $\triangle (\tilde{p}\tilde{x}\tilde{y})$ emanating from 
$x$ to $y$ and from $\tilde{x}$ to $\tilde{y}$, respectively.
By Corollary \ref{2013_02_13_cor2.1} and Lemma \ref{2013_02_16_lem3.1}, 
$d \big( p,c(s) \big) >\rad_p$ holds for all $s \in (0,d(x,y))$. 
This contradicts the definition of $\rad_p$.
The second assertion follows from Lemma \ref{2013_02_17_lem3.2}. 
$\qedd$
\end{proof}

We say that a point $x \in M$ is a {\em (forward) critical point} for $p \in M$
if, for every $w \in T_{x} M \setminus \{0\}$, 
there exists $v \in \cG_p (x)$ such that $g_v (v, w) \le 0$ 
(see \eqref{G_p} for the definition of $\cG_p (x)$).
Then, we may prove Gromov's isotopy lemma \cite{G} by 
similar arguments to the Riemannian case:

\begin{lemma}\label{2013_02_17_lem3.4}
Given $0 < r_1 < r_2 \le \infty$, 
if $\ol{B_{r_2}^+(p)} \setminus B_{r_1}^+(p)$ has no critical point for $p \in M$, 
then $\ol{B_{r_2}^+(p)} \setminus B_{r_1}^+(p)$ is homeomorphic to 
$\partial B_{r_1}^+(p) \times [r_1, r_2 ]$.
\end{lemma}

\begin{lemma}\label{2013_02_17_lem3.5}
There are no critical point for $p$ in $\ol{B^+_{\frac{\pi}{2}}(p)} \setminus \{ p \}$. 
In particular, $\ol{B^+_{\frac{\pi}{2}}(p)}$ is a topological disk.
\end{lemma}

\begin{proof}
By Lemma \ref{2013_02_17_lem3.2}, 
$\partial B^+_{\frac{\pi}{2}}(p)$ has no critical point for $p$.
Suppose that there exists a critical point $x \in B^+_{\frac{\pi}{2}}(p) \setminus \{ p \}$ for $p$.
Let $q \in M$ be the same as that in Lemma \ref{2013_02_17_lem3.3} such that $d(p, q) = \rad_p$, 
and $c:[0, a] \lra M$ a unit speed minimal geodesic 
segment emanating from $q$ to $x$, where $a:= d(q,x)$. 
Then, $c ( [0,a]) \cap \partial B_{\frac{\pi}{2}}^+ (p) \not= \emptyset$. 
From the cases in the proof of Lemma \ref{2013_02_17_lem3.1} and Lemma \ref{2013_02_17_lem3.2}, 
it is sufficient to consider the case where $c ( [0,a] ) \cap \partial B_{\frac{\pi}{2}}^+ (p)$ is one point, say 
\[
\{ q_1 \}:= c ( [0,a]) \cap \partial B_{\frac{\pi}{2}}^+ (p).
\]
Since both $q = c(0)$ and $x = c(a)$ are critical points for $p$, we have 
\begin{equation}\label{2013_02_17_lem3.5-1}
\ora{\angle} (pc (0)c(a)) \le \frac{\pi}{2}, \qquad 
\ola{\angle} (pc (a)c(0)) \le \frac{\pi}{2}.
\end{equation}
Note that $c$ does not pass through $p$, 
because, by the definition of critical points, 
there exist at least two minimal segments emanating from $p$ to $x$ and 
$c$ is minimal. Now, take a subdivision $s_0 := 0 < s_1 < \cdots < s_{k -1} <s_k :=  a$ 
of $[0, a]$ such that $c(s_1)=q_1 \in \partial B^+_{\frac{\pi}{2}}(p)$ and that 
$\triangle (\ora{p c (s_{i-1})}, \ora{p c (s_i)})$ 
admits a comparison triangle 
$\wt{\triangle}^i := \triangle (\tilde{p} \wt{c (s_{i-1})} \wt{c (s_i)}) \subset \Sph^2$
for each $i = 2,3, \ldots, k$. Note that $\triangle (\ora{p c (s_0)}, \ora{p c (s_1)})$ 
admits its a comparison triangle 
$\wt{\triangle}^1 := \triangle (\tilde{p} \wt{c (s_0)} \wt{c (s_1)}) \subset \Sph^2$. 
Applying Lemma \ref{2013_02_13_lem2.1} 
to $\triangle (\ora{p c (s_{i-1})}, \ora{p c (s_i)})$ for each $i = 3, 4, \ldots, k$, 
\begin{equation}\label{2013_02_17_lem3.5-2}
\ora{\angle} c (s_{i-1})  
\ge 
\angle\big( \tilde{p} \wt{c (s_{i-1})} \wt{c (s_i)} \big), \qquad 
\ola{\angle}c (s_i)
\ge 
\angle \big( \tilde{p} \wt{c (s_i)} \wt{c (s_{i-1})} \big). 
\end{equation}
In particular, by \eqref{2013_02_17_lem3.5-1} and \eqref{2013_02_17_lem3.5-2}, we have 
\begin{equation}\label{2013_02_17_lem3.5-2.5}
\angle\big( \tilde{p} \wt{c (s_k)} \wt{c (s_{k-1})} \big) \le \frac{\pi}{2}.
\end{equation}
In cases where $i= 1,2$, it follows from the limit argument by using \cite[Proposition 2.1]{TS}, 
which is the technic in the proof of Lemma \ref{2013_02_17_lem3.1}, that 
\begin{equation}\label{2013_02_17_lem3.5-3}
\ora{\angle} c (s_0)  
\ge 
\angle\big( \tilde{p} \wt{c (s_0)} \wt{c (s_1)} \big), \qquad 
\ola{\angle}c (s_1)
\ge 
\angle \big( \tilde{p} \wt{c (s_1)} \wt{c (s_0)} \big), 
\end{equation}
and that 
\begin{equation}\label{2013_02_17_lem3.5-4}
\ora{\angle} c (s_1)  
\ge 
\angle\big( \tilde{p} \wt{c (s_1)} \wt{c (s_2)} \big), \qquad 
\ola{\angle}c (s_2)
\ge 
\angle \big( \tilde{p} \wt{c (s_2)} \wt{c (s_1)} \big). 
\end{equation}
In particular, by \eqref{2013_02_17_lem3.5-1} and \eqref{2013_02_17_lem3.5-3}, we have 
\begin{equation}\label{2013_02_21_lem3.5-1}
\angle\big( \tilde{p} \wt{c (s_0)} \wt{c (s_1)} \big) \le \frac{\pi}{2}.
\end{equation}
Starting from $\wt{\triangle}^1$, we inductively draw a geodesic triangle 
$\wt{\triangle}^{i + 1} \subset \Sph^2$ which is adjacent to $\wt{\triangle}^i$ 
so as to have a common side
$\tilde{p} \wt{c (s_i)}$, where 
$0 := \theta (\wt{c (s_0)}) \le \theta (\wt{c (s_1)}) \le \cdots \le \theta (\wt{c (s_k)})$.
Since $\ola{\angle} c (s_i) + \ora{\angle} c (s_i) \le \pi$ for each $i = 1,2, \ldots, k-1$,  
we obtain, by \eqref{2013_02_17_lem3.5-2}, \eqref{2013_02_17_lem3.5-3}, \eqref{2013_02_17_lem3.5-4}, 
\begin{equation}\label{2013_02_17_lem3.5-5}
\angle\big( \tilde{p} \wt{c (s_i)} \wt{c (s_{i-1})} \big) + 
\angle\big( \tilde{p} \wt{c (s_i)} \wt{c (s_{i+1})} \big) 
\le \pi.
\end{equation}
Let $\wh{\xi}:[0, L_{\rm m} (c)] \lra \Sph^2$ denote the broken geodesic segment 
consisting of minimal geodesic segments from $\wt{c (s_{i-1})}$ to $\wt{c (s_i)}$,
$i=1,2,\ldots,k$.
Set $\wh{\xi} (r) := (t(\wh{\xi} (r)), \theta (\wh{\xi} (r)))$. 
By \eqref{2013_02_17_lem3.5-5}, we have the unit speed geodesic 
$\wt{\eta} :[0, b] \lra \Sph^2$ emanating from $\wt{\eta} (0) = \wt{c (s_0)}$ to $\wt{\eta} (b) = \wt{c (s_k)}$ 
and passing under $\wh{\xi} ( [0, L_{\rm m} (c)] )$, i.e., 
$\theta (\wt{\eta}) \in [0, \theta (\wt{c (s_k)})]$ on $[0, b]$ 
and $t(\wh{\xi} (r)) > t(\wt{\eta} (u))$ for all $(r,u) \in (0,L_{\rm m} (c)) \times (0, b)$ with 
$\theta (\wh{\xi} (r)) = \theta (\wt{\eta} (u))$. 
Since $b \le L_{\rm m} (c) < \pi$ by \eqref{2013_05_19_Sec3.1}, 
$\wt{\eta}$ is minimal with $\angle (\dot{\wt{\gamma}}(0), \dot{\wt{\sigma}}(0)) < \pi$, 
where $\wt{\gamma}$ and $\wt{\sigma}$ denote minimal geodesic segments (i.e., sub-meridians) 
emanating from $\tilde{p}$ to $\wt{c(s_0)}$ and from $p$ to $\wt{c(s_k)}$, respectively. 
Since $\wt{\eta}$ lives under $\wh{\xi} ( [0, L_{\rm m} (c)] )$, we have, 
by \eqref{2013_02_17_lem3.5-2.5} and \eqref{2013_02_21_lem3.5-1}
\begin{equation}\label{eq-2013_04_18}
\angle (\dot{\wt{\eta}}(0), - \dot{\wt{\gamma}}(\tilde{d}(\tilde{p}, \wt{c(s_0)}))) \le \frac{\pi}{2}, \qquad 
\angle (\dot{\wt{\eta}}(b), \dot{\wt{\sigma}}(\tilde{d}(\tilde{p}, \wt{c(s_k)}))) \le \frac{\pi}{2}.
\end{equation}
Since $\tilde{d} (\tilde{p}, \wt{c(s_0)}) > \pi/2 > \tilde{d} (\tilde{p}, \wt{c(s_k)})$, 
there exists $b_0 \in (0, b)$ such that $\wt{\eta}(b_0) \in \partial B_{\frac{\pi}{2}}(\tilde{p})$. 
Note that the three points $\tilde{p}, \wt{\eta} (0)$, $\wt{\eta} (b)$ are not contained 
in the same great circle, because of \eqref{eq-2013_04_18}. Let $\ol{\eta}: [0, \pi] \lra \Sph^2$ be 
an extension of $\wt{\eta}$ to the antipodal point $\wt{c(s_0)}_\pi$ to $\wt{c (s_0)} = \wt{\eta} (0)$, 
and set $\ol{\eta} (u):=(t(u), \theta (u))$. 
Since $\Sph^2 \setminus B_{u}(\ol{\eta} (0))$ is strictly convex 
for all $u \in (\pi/2,\pi)$ (by the same proof of Lemma \ref{2013_02_16_lem3.1}), 
$\angle (\dot{\ol{\eta} } (\rad_p), (\partial / \partial t)|_{\ol{\eta}(\rad_p)} ) > \pi / 2$ holds.
This implies 
\begin{equation}\label{eq1_2013_05_20}
t'(\rad_p) <0, 
\end{equation}
where note that $\ol{\eta}$ emanates from $\wt{c(s_0)}$ to $\wt{c(s_0)}_\pi$. 
Since $\ol{\eta} ((b_0, \rad_p)) \subset B_{\frac{\pi}{2}}(\tilde{p})$ and 
$f' (t) = \cos t> 0$ on $(0, \pi/2)$, it follows from \cite[(7.1.15)]{SST} that 
\begin{equation}\label{eq2_2013_05_20}
t''(u) = f (t(u))f'(t(u)) \theta'( t(u))^2 >0
\end{equation}
holds on $(b_0, \rad_p)$. Hence, by \eqref{eq1_2013_05_20} and \eqref{eq2_2013_05_20}, 
$t(u)$ is decreasing on $[b_0, \rad_p]$. Since $b \in (b_0, \rad_p)$,  
\[
\angle (\dot{\wt{\eta}}(b), \dot{\wt{\sigma}}( t(b) )) > \frac{\pi}{2}.
\] 
This contradicts the right inequality in \eqref{eq-2013_04_18}. Therefore, $\ol{B_{\frac{\pi}{2}}^+(p)} \setminus \{ p \}$ has 
no critical point for $p$. 
By Lemma \ref{2013_02_17_lem3.4}, $\ol{B^+_{\frac{\pi}{2}}(p)}$ is a topological disk.
$\qedd$
\end{proof}

By Lemmas \ref{2013_02_17_lem3.3} and \ref{2013_02_17_lem3.5}, 
$M$ is homeomorphic to $\Sph^n$.
$\qedd$

\section{Appendix: Proof of Lemma \ref{2013_02_13_lem2.1}}\label{sec4}

Let $(M,F,p)$ be a forward complete, connected $C^\infty$-Finsler manifold with a base point 
$p \in M$, and let $d$ denote its distance function and $\Cut (p)$ the cut locus of $p$. 
Set $B^-_r(x):=\{ y \in M \,|\, d(y,x)<r \}$. 
Take a point $q \in M \setminus (\Cut (p) \cup \{p\})$ and small $r>0$ such that
$B^-_{2r}(q) \cap (\Cut (p) \cup \{p\} )= \emptyset$ 
and that $B_r^{\pm}(q):=B^+_r(q) \cap B^-_r(q)$ is geodesically convex 
(i.e., any minimal geodesic joining two points in $B_r^{\pm}(q)$ is contained in $B_r^{\pm}(q)$).
Given a unit speed minimal geodesic segment $c:(-\ve,\ve) \lra B_r^{\pm}(q)$,
we consider the $C^{\infty}$-variation 
\[
\varphi (t, s) := \exp_{p}\left( \frac{t}{l} \exp_p^{-1}\big( c(s) \big) \right),
\qquad (t,s) \in [0,l] \times (-\ve,\ve),
\]
where $l:=d(p,c(0))$.
Since $x:=c(0) \not\in \Cut(p)$, there is a unique minimal geodesic segment
$\gamma:[0,l] \lra M$ emanating from $p$ to $x$.
By setting 
$J(t) := \frac{\partial \varphi}{\partial s} (t, 0)$, 
we get the Jacobi field $J$ along $\gamma$ with $J(0) = 0$ and $J (l) = \dot{c}(0)$. 
Note that $J(t) \not= 0$ on $(0, l\,]$ from the minimality of $\gamma$, and that  
\begin{equation}\label{2013_02_14_J}
J^{\perp} (t) := J(t) - \frac{g_{\dot{\gamma}(l)} (\dot{\gamma}(l),\dot{c}(0))}{l}t\dot{\gamma}(t), 
\qquad t \in [0, l], 
\end{equation}
is the $g_{\dot{\gamma}}$-orthogonal component $J^{\perp}(t)$
to $\dot{\gamma}(t)$ (see \cite[Lemma 3.2]{KOT1}). 
Moreover, since $\gamma$ is unique, 
it follows from the proof of \cite[Lemma 2.2]{KOT1} that 
\begin{equation}\label{2013_02_14_eq2.1}
-\cos \ora{\angle}\big( pxc(\ve) \big) =\cos \ola{\angle}\big( pxc(-\ve) \big)
=\lambda^{-1} g_{\dot{\gamma}(l)} \big( \dot{\gamma}(l),\dot{c}(0) \big),
\end{equation}
where $\lambda:=\max\{1,F(-\dot{c}(0))\}$. 
Hence, $\pi-\ora{\angle}(pxc(\ve))=\ola{\angle}(pxc(-\ve))$. 
In the following discussion, we set
\begin{equation}\label{2013_02_13_eq2.1}
\omega:=\pi-\ora{\angle}(pxc(\ve))=\ola{\angle}(pxc(-\ve)).
\end{equation}

Hereafter, we assume that the radial flag curvature of $(M,F,p)$
is bounded below by $1$. Hence, its model surface is the unit sphere $(\Sph^2, \tilde{p})$ 
with its metric $d\tilde{s}^2 = dt^2 + f(t)^2d \theta^2$, $(t,\theta) \in (0,\pi) \times \Sph_{\tilde{p}}^1$, 
such that $f(t) = \sin t$. 
For small $\delta>0$ with $\delta <1$, we set 
\[
f_{\delta} (t) := \frac{1}{\sqrt{1- \delta}} \sin (\sqrt{1-\delta}\,t) 
\]
on $[0, \pi / \sqrt{1- \delta}\,]$. 
Then, $f_{\delta}$ satisfies $f''_{\delta} + (1 - \delta) f_{\delta}= 0$ 
with $f_{\delta}(0) = 0$, $f_{\delta}'(0) = 1$. 
Thus, we have a new sphere $(\Sph^2_\delta, \tilde{o})$ 
with the metric $d\tilde{s}^2_{\delta} = dt^2 + f_{\delta}(t)^{2} d\theta^2$
on $(0, \pi / \sqrt{1- \delta}\,) \times \Sph_{\tilde{o}}^{1}$.
Since the curvature $1-\delta$ of $(\Sph^2_\delta, \tilde{o})$ is less than $1$,
we may also employ $(\Sph^2_\delta, \tilde{o})$ as a reference surface for $M$.

Let $c$, $x=c(0)$, $\gamma$ and $l=d(p,x)$ be the same in the above.
Fix a point $\tilde{x} \in \Sph^2_\delta$ with $ \tilde{d}_\delta(\tilde{o}, \tilde{x})= l$,
where $\tilde{d}_\delta$ denotes the distance function induced from $d\tilde{s}^2_{\delta}$. 
Let $\wt{\gamma} : [0,l] \lra \Sph^2_\delta$ be the minimal geodesic segment
from $\tilde{o}$ to $\tilde{x}$, and take a unit parallel vector field $\wt{E}$
along $\wt{\gamma}$ orthogonal to $\dot{\wt{\gamma}}$.
Define the Jacobi field $\wt{X}$ along $\wt{\gamma}$ by 
\begin{equation}\label{2013_02_14_eq2.2}
\wt{X}(t) := \frac{1}{f_{\delta}(l)} f_{\delta}(t) \wt{E}(t). 
\end{equation}

\begin{lemma}{\rm (\cite[Lemma 3.4]{KOT1})}\label{2013_02_13_lem2.2}
For any Jacobi field $X$ along $\gamma$ which is $g_{\dot{\gamma}}$-orthogonal to
$\dot{\gamma}$ and satisfies $X(0) = 0$ and $g_{\dot{\gamma}(l)} (X (l), X(l)) = 1$, we have
\[
\wt{I}_{l} (\wt{X}, \wt{X}) 
\ge I_{l}(X, X)
+ \frac{\delta}{f_{\delta}(l)^{2}} \int_{0}^{l} f_{\delta}(t)^{2}\,dt.
\]
Here, $I_l$ and $\wt{I}_{l}$ denote the index forms with respect to 
$\gamma|_{[0, \,l]}$ and $\wt{\gamma}|_{[0, \,l]}$, respectively. 
\end{lemma}

Fix a geodesic $\tilde{c}:(-\ve,\ve) \lra \Sph^2_{\delta}$
with $\tilde{c}(0)=\tilde{x}$ such that
\[
\angle\big( \dot{\wt{\gamma}}(l),\dot{\tilde{c}}(0) \big) =\omega,
 \qquad \|\dot{\tilde{c}}\|=\lambda:=\max\left\{ 1,F\!\left( -\dot{c}(0) \big) \right. \right\},
\]
where $\omega$ is as that in \eqref{2013_02_13_eq2.1}. 
Consider the geodesic variation
\[
\tilde{\varphi}(t,s):=\exp_{\tilde{o}}\left( \frac{t}{l}\exp^{-1}_{\tilde{o}}\big( \tilde{c}(s) \big) \right),
 \qquad (t,s) \in [0,l] \times (-\ve,\ve).
 \]
By setting $\wt{J}(t) := \frac{\partial \wt{\varphi}}{\partial s} (t, 0)$, 
we get the Jacobi field $\wt{J}$ 
along $\wt{\gamma}$ with $\wt{J}(0) = 0$ and $\wt{J}(l) = \dot{\tilde{c}}(0)$. 
And the Jacobi field 
\[
\wt{J}^{\perp} (t) 
:= 
\wt{J}(t) 
- 
\frac{\langle 
\dot{\wt{\gamma}}(l), \dot{\tilde{c}}(0)
\rangle}{l} t \dot{\wt{\gamma}}(t)
\]
along $\wt{\gamma}$ is orthogonal to $\dot{\wt{\gamma}}(t)$ on $[0,l]$. 

\begin{lemma}\label{2013_02_13_lem2.3}
Assume that 
\begin{enumerate}[$(1)$]
\item
$B^-_{2r}(q) \subset B^+_{\frac{\pi}{2}}(p)$,
\item 
$F(v)^2 \ge g_{\dot{\gamma}(l)}(v,v)$ for all $v \in T_xM$.
\end{enumerate}
If $\omega \in (0, \pi)$, then there exists $\delta_{1}:= \delta_{1}(f,r)> 0$ such that, 
for any $\delta \in (0, \delta_{1})$, 
\[
\wt{I}_{l} (\wt{J}^{\perp}, \wt{J}^{\perp}) 
-I_{l} (J^{\perp}, J^{\perp}) 
\ge \delta \,C_1\,g_{\dot{\gamma}(l)}(J^{\perp}(l), J^{\perp}(l))>0
\]
holds, where  
$C_1 := \frac{1}{2 f(l_{0})^{2}} \int_{0}^{l_{0}} f(t)^{2} \,dt$ and $l_0:=d(p,q)$. 
\end{lemma}

\begin{proof}
By the assumption (2) in this lemma, 
\begin{equation}\label{2013_02_13_lem2.3-10}
\lambda^2 \ge g_{\dot{\gamma}(l)}(\dot{c}(0),\dot{c}(0)).
\end{equation}
Indeed, \eqref{2013_02_13_lem2.3-10} is immediate in the case where $\lambda =1$. 
If $\lambda = F(-\dot{c}(0))$, then
\[
1 \ge g_{\dot{\gamma}(l)}\left( \frac{-\dot{c}(0)}{F(-\dot{c}(0))}, \frac{-\dot{c}(0)}{F(-\dot{c}(0))} \right) 
= \frac{1}{F(-\dot{c}(0))^2}g_{\dot{\gamma}(l)}(\dot{c}(0),\dot{c}(0)).
\]

By \eqref{2013_02_14_eq2.1} and \eqref{2013_02_13_eq2.1}, 
$g_{\dot{\gamma}(l)} \big(  \dot{\gamma}(l),\dot{c}(0) \big)= \lambda \cos \omega$.
Then, $\wt{J}^{\perp} (l) = \pm \lambda \sin \omega \cdot \wt{X}(l)$ holds,
where $\wt{X}$ is the same as that in \eqref{2013_02_14_eq2.2}.
Since both $\wt{J}^{\perp}$ and $\wt{X}$ are Jacobi fields on $\Sph^2_\delta$, 
$\wt{J}^{\perp}(t) = \pm \lambda \sin \omega \cdot \wt{X}(t)$ on $[0, l]$. 
Hence
\begin{equation}\label{2013_02_13_lem2.3-2}
\wt{I}_{l} (\wt{J}^{\perp}, \wt{J}^{\perp}) 
= (\lambda \sin\omega)^2 \wt{I}_{l} (\wt{X}, \wt{X}).
\end{equation}
On the other hand, it follows from \eqref{2013_02_14_J} and \eqref{2013_02_13_lem2.3-10} that 
\[
g_{\dot{\gamma}(l)} \big( J^{\perp}(l),J^{\perp}(l) \big) = 
g_{\dot{\gamma}(l)} \big( \dot{c} (0), \dot{c} (0) \big) -(\lambda \cos \omega)^2
\le (\lambda \sin \omega)^2.
\]
Then, we get a constant 
$a:=(\lambda \sin \omega)^2 - g_{\dot{\gamma}(l)} \big( J^{\perp}(l),J^{\perp}(l) \big) \ge 0$. 
Since $g_{\dot{\gamma}(l)} \big( J^{\perp}(l),J^{\perp}(l) \big) >0$ for $\omega \in (0, \pi)$, 
we have, by Lemma \ref{2013_02_13_lem2.2}, 
\[
\wt{I}_l(\wt{X},\wt{X}) 
\ge \frac{I_l(J^{\perp},J^{\perp})}{g_{\dot{\gamma}(l)} \big( J^{\perp}(l),J^{\perp}(l) \big)}
 +\frac{\delta}{f_{\delta}(l)^2} \int_0^l f_{\delta}(t)^2 \,dt, 
\]
hence
\begin{align}\label{2013_02_13_lem2.3-3}
- I_l(J^{\perp},J^{\perp}) 
&\ge
-g_{\dot{\gamma}(l)} \big( J^{\perp}(l),J^{\perp}(l) \big) 
\left\{
\wt{I}_l(\wt{X},\wt{X}) 
-\frac{\delta}{f_{\delta}(l)^2} \int_0^l f_{\delta}(t)^2 \,dt
\right\}\\
&= \{a -(\lambda \sin \omega)^2\}\wt{I}_l(\wt{X},\wt{X}) 
+ 
\frac{\delta \cdot g_{\dot{\gamma}(l)} \big( J^{\perp}(l),J^{\perp}(l) \big) }{f_{\delta}(l)^2} \int_0^l f_{\delta}(t)^2 \,dt.\notag
\end{align}
By \eqref{2013_02_13_lem2.3-2} and \eqref{2013_02_13_lem2.3-3}, 
\begin{align*}
\wt{I}_l(\wt{J}^{\perp},\wt{J}^{\perp}) -I_l(J^{\perp},J^{\perp})
&\ge a\wt{I}_l(\wt{X},\wt{X})
 +\frac{\delta \cdot g_{\dot{\gamma}(l)} \big( J^{\perp}(l),J^{\perp}(l) \big) }{f_{\delta}(l)^2} \int_0^l f_{\delta}(t)^2\\
&\ge \frac{\delta \cdot g_{\dot{\gamma}(l)} \big( J^{\perp}(l),J^{\perp}(l) \big) }{f_{\delta}(l)^2} \int_0^l f_{\delta}(t)^2, 
\end{align*}
where note that $a \ge 0$, and that 
$
\wt{I}_l(\wt{X},\wt{X})
=\frac{\sqrt{1-\delta}}{\tan(\sqrt{1 -\delta}l)}>0$,
because $l < \pi/2 < \pi / 2 \sqrt{1-\delta}$ by the assumption (1) in this lemma. 
Since $|l-l_0| \le \max\{d(q,x),d(x,q)\}<r$, and since $l,l_0< \pi/2$ (from the (1)),
taking smaller $\delta_1(f,r)>0$ if necessary, we get the desired assertion in this lemma 
for all $\delta \in (0, \delta_{1})$. 
$\qedd$
\end{proof}

\begin{lemma}\label{2013_02_14_lem2.1}
Assume that 
\begin{enumerate}[$(1)$]
\item
$B^-_{2r}(q) \subset B^+_{\frac{\pi}{2}}(p)$,
\item 
$F(v)^2 \ge g_{\dot{\gamma}(l)}(v,v)$ for all $v \in T_xM$, 
\item
$\cT_M(\dot{\gamma}(l), \dot{c}(0)) = 0$.
\end{enumerate}
For each $\delta \in (0, \delta_{1})$, $\theta \in (0, \pi/2)$, if 
$\omega \in [\theta, \pi -\theta]$, then  
there exists $\ve' :=\ve'(M,l,f,\ve,\delta,\theta)$ $\in (0, \ve)$ such that
$L(s) \le \wt{L}(s)$ holds for all $s \in [-\ve', \ve']$. Here, $L(s):=d(p,c(s))$ 
and $\wt{L}(s):=\tilde{d}_{\delta}(\tilde{o},\tilde{c}(s))$.
\end{lemma}

\begin{proof}
We will state the outline of the proof, since the proof is very similar to 
\cite[Lemma 3.6]{KOT1} thanks to Lemma \ref{2013_02_13_lem2.3}. 
Set $\cR(s) := L(s) - \left\{ L(0) + L'(0)s + L''(0)s^{2}/2 \right\}$. 
Then, there exists $C_2:=C_2(M,l)>0$ such that 
\[
L(s) 
= 
L(0) + L'(0)s + \frac{1}{2} L''(0)s^{2} +\cR (s)
\le 
l+s \lambda \cos \omega + \frac{s^{2}}{2} I_{l} (J^{\perp}, J^{\perp}) +C_2 |s|^3.
\]
Note that $L'(0) = \lambda \cos \omega$ and $L''(0) = I_{l} (J^{\perp}, J^{\perp})$ 
hold by \cite[Lemma 3.3]{KOT1}, \eqref{2013_02_14_eq2.1}, 
\eqref{2013_02_13_eq2.1}, and the assumption (3) in this lemma. 
Similarly, 
\[ \wt{L}(s) \ge l+s \lambda \cos \omega + \frac{s^{2}}{2} 
\wt{I}_{l} (\wt{J}^{\perp}, \wt{J}^{\perp}) -C_3 |s|^{3} \]
for some $C_3:=C_3(f,l)>0$ and all $s \in (-\ve,\ve)$.
Since $g_{\dot{\gamma}(l)}(J^{\perp}(l), J^{\perp}(l)) > 0$ for all $\omega \in [\theta, \pi -\theta]$, 
there exists $C_4:= C_4(M, \theta) >0$ such that 
$g_{\dot{\gamma}(l)}(J^{\perp}(l), J^{\perp}(l)) > C_4> 0$. 
From Lemma \ref{2013_02_13_lem2.3}, 
$\wt{L}(s) - L(s) \ge s^2\{ \delta C_1 C_4 -2(C_2+C_3)s \} /2$ holds.
Therefore, we get $L(s) \le \wt{L}(s)$ for all $s \in [-\ve',\ve']$, 
if $\ve' := \min \left\{ \ve, \delta C_{1} C_4 / 2(C_2+C_3) \right\}$.
$\qedd$
\end{proof}

Thanks to Lemma \ref{2013_02_14_lem2.1} and the structure of $\Sph^2_\delta$, 
we may prove Lemma \ref{2013_02_13_lem2.1} by the same arguments in Sections $4$, $5$, 
and $6$ in \cite{KOT1}. $\qedd$

\begin{remark} Although we do not consider cases of $\omega=0$, or $\pi$ in Lemma \ref{2013_02_14_lem2.1}, Lemma \ref{2013_02_13_lem2.1} holds in cases of 
$\ora{\angle}x=\pi$, $\ola{\angle}y=0$, or $\ora{\angle}x=0$, $\ola{\angle}y=\pi$ 
because the reverse curve $\bar{c}$ of the geodesic segment $c$ is geodesic.
\end{remark}

\begin{flushleft}
K.~Kondo,\\
Department of Mathematics, Tokai University,\\
Hiratsuka City, Kanagawa Pref. 259-1292, Japan\\
{\small e-mail: {\tt keikondo@keyaki.cc.u-tokai.ac.jp}}
\end{flushleft}

\end{document}